\documentclass[reqno]{amsart}

\newtheorem{theorem}{Theorem}[section]

\theoremstyle{definition}

\usepackage[T1]{fontenc}
\usepackage[utf8]{inputenc}
\usepackage{lmodern}
\usepackage{textcomp}
\usepackage{microtype}
\usepackage{tikz}
\usepackage{mathtools, bm, nccmath}

\newcounter{proofstep}

\usepackage[english]{babel}
\usepackage{csquotes}
\usepackage[export]{adjustbox}
\usepackage{circuitikz}

\usepackage{tabularx}

\usepackage{graphicx}
\usepackage{setspace}

\usepackage{amsmath}
\usepackage{amsfonts}
\usepackage{amssymb}
\usepackage{amsthm}
\usepackage{xcolor}%
\usepackage{soul}
\usepackage{graphicx} 
\usepackage{lipsum}
\numberwithin{equation}{section}

\newcommand*{\N}{\mathbb{N}}
\newcommand*{\Z}{\mathbb{Z}}

\newcommand*{\R}{\mathbb{R}}
\newcommand*{\C}{\mathbb{C}}

\newcommand*{\flap}{(-\Delta_g)^{\frac{\alpha}{2}}}

\usepackage{thmtools}

\declaretheorem[
	name=Remark,
	style=remark,
	numbered=no
	]{rem}

\newcounter{thmletter}

\usepackage[pdftex,pdfpagelabels]{hyperref}

\numberwithin{equation}{section}

\newcommand\restr[2]{{
  \left.\kern-\nulldelimiterspace 
  #1 
  \littletaller 
  \right|_{#2} 
  }}

\newcommand{\littletaller}{\mathchoice{\vphantom{\big|}}{}{}{}}

\makeatletter
\@namedef{subjclassname@2020}{\textup{2020} Mathematics Subject Classification}
\makeatother

\calclayout

\begin{document}
\title[Eigenvalue bounds for non-self-adjoint Schr\"odinger operators]
{Eigenvalue bounds for non-self-adjoint Schr\"odinger operators and pseudodifferential generalizations}

\author[E. Stefanescu]{Eduard Stefanescu}
\address{Institut f\"ur Analysis und Zahlentheorie, TU Graz, Steyrergasse 30, 8010 Graz, Austria}
\email{\href{mailto:eduard.stefanescu@tugraz.at}{eduard.stefanescu@tugraz.at}}


\begin{abstract}
This is mostly a survey paper, where we collect results concerning the spectral bounds of deterministic and random Schr\"odinger operators with complex potentials, both on \(\mathbb{R}^d\) and on compact manifolds. The survey part is complemented by a new theorem, where we extend the result on spectral bounds on compact manifolds to the case of fractional Laplacians, applying methods by Cuenin and Sogge. These bounds are formulated in terms of the \(L^p\)-norms of the corresponding potentials. 
\end{abstract}

\maketitle

\section{Introduction}

In this paper, we study eigenvalue bounds for  Schr\"odinger operators $H:=-\Delta+V$
on $\mathbb{R}^d$ as well as on compact manifolds, where the bounds are expressed in terms of $L^p$-norms of the possibly complex-valued potential $V$.

When the potential $V : \mathbb{R}^d \to \mathbb{R}$ is real-valued and sufficiently regular (e.g. $V\in L_{\text{loc}}^{d/2}$),
the operator $H$ is self-adjoint on a natural domain (e.g. $H^1(\R^d)$) in $L^2(\mathbb{R}^d)$.
Its spectral properties admit a clear physical interpretation:
the spectrum is real, the time evolution is unitary, and energy is conserved.

By contrast, allowing \(V\) to be complex-valued fundamentally alters the spectral
picture, as self-adjointness of the Hamiltonian \(H\) is lost.
This loss gives rise to complex spectra, spectral instability, the breakdown
of variational principles, and necessitates different techniques.

A main new result of this paper is an extension of eigenvalue bounds of this type to fractional Schr\"odinger operators on closed Riemannian manifolds. More precisely, we prove spectral inclusion bounds for operators of the form \((-\Delta_g)^{\alpha/2}+V\), where \(V\in L^q(M)\) may be complex-valued, thereby generalizing the compact-manifold result by Cuenin \cite{CJC} from the classical Laplacian to the fractional setting.

The paper is organized as follows. In Section~\ref{appA} we introduce notation and collect the basic analytic tools used throughout the paper. Section \ref{sec3} briefly reviews early motivations for non-self-adjoint Schr\"odinger operators in physics and mathematics. Section \ref{sec4} recalls classical eigenvalue bounds in the self-adjoint setting, including Keller-type estimates, the Birman--Schwinger principle, Schatten--von Neumann ideals, and Lieb--Thirring-type inequalities. Section \ref{sec5} then turns to the non-self-adjoint setting and surveys deterministic and random eigenvalue bounds for complex decaying potentials on \(\mathbb R^d\), as well as deterministic and forthcoming random results on compact manifolds. In Section \ref{sec6} we state our extension to fractional Laplacians on compact manifolds, and in Section \ref{sec7} we give a short proof-idea. The full proof, including broader extensions and additional results, will be published by the author in a forthcoming publication.

In Appendix~\ref{app} we collect some operator-theoretic background for the compact-manifold setting, including the compactness of the resolvent of the free Laplacian, the resulting discreteness of the spectrum, and the corresponding extension to elliptic self-adjoint pseudodifferential operators of positive order, such as the fractional Laplacian. We also briefly recall that for bounded potentials \(V\in L^\infty(M)\), the operator \((-\Delta_g)^{\alpha/2}+V\) can be defined by quadratic form methods as an \(m\)-sectorial operator.

\section{Notation and Preliminaries}\label{appA}
Throughout this paper, the dimension $d \geq 1$ is assumed to be fixed. We denote by \(D(c,r)\subset \mathbb{C}\) the closed disc centered at \(c\) with radius \(r\). For $A,B\in\R$ we write $A\lesssim B$, whenever there is a constant $c\ge0$ such that $A\leq cB$. If $A\lesssim B\lesssim A$, we write $A\simeq B$.

Let $\Omega \subset \mathbb{R}^d$. For $1 \le p < \infty$ we define
\[
L^p(\Omega)
:=
\left\{
f \text{ measurable} : \|f\|_{L^p(\Omega)} < \infty
\right\},
\qquad
\|f\|_{L^p(\Omega)}^p
:=
\int_\Omega |f|^p \, d\mu,
\]
and
\[
L^\infty(\Omega)
=
\left\{
f:\Omega\to\mathbb{C}\text{ measurable} :
\operatorname*{ess\,sup}_{x\in\Omega}|f(x)|<\infty
\right\}.
\]

We utilize Einstein summation notation. For a fixed Riemannian manifold $M$ with metric $g$ we define the Laplace Beltrami operator as
\[
\Delta_g
:=
\frac{1}{\sqrt{|g|}}
\partial_i\!\left(
\sqrt{|g|}\, g^{ij}\, \partial_j
\right).
\]
Hence in $\mathbb{R}^d$ we have $\Delta=\partial_j\partial^j.$
For spectral-properties on the Laplace Beltrami operator see Appendix \ref{app}.

By abuse of notation we write $H=-\Delta+V$, where $V$ is the multiplication operator by the function denoted by the same symbol $V(x)$, i.e.
\[
(Vf)(x)=V(x)f(x).
\]


Let $M$ be compact and \((e_j)_{j=0}^\infty\) be an orthonormal basis of \(L^2(M)\) consisting of eigenfunctions of \(-\Delta_g\), with corresponding eigenvalues \(\lambda_j^2\).
For \(\alpha>0\), the fractional Laplacian \((-\Delta_g)^{\alpha/2}\) is defined via spectral calculus by
\[
(-\Delta_g)^{\alpha/2}f
:=
\sum_{j=0}^\infty \lambda_j^\alpha \langle f,e_j\rangle e_j,
\]
for all \(f\) in the domain
\[
\mathcal{D}\bigl((-\Delta_g)^{\alpha/2}\bigr)
=
\left\{
f\in L^2(M):
\sum_{j=0}^\infty \lambda_j^{2\alpha} |\langle f,e_j\rangle|^2<\infty
\right\}.
\]

\section{Early works on non-self-adjoint Schr\"odinger operators}\label{sec3}

\subsection{Early motivation in physics}

From the physical perspective, non-self-adjoint Schrödinger operators, e.g. with complex potentials, arise naturally in quantum
mechanics as effective descriptions of open systems.
Decaying complex potentials already appear in the seminal work of
Gamow~\cite{Gamow}, where complex energies are used to model radioactive decay.
One of the earliest systematic investigations of spectral problems without the
assumption of self-adjointness is due to Siegert~\cite{Siegert}, who observed that
imposing outgoing boundary conditions leads naturally to eigenvalue problems with
complex eigenvalues.

\subsection{Early motivation in mathematics}

On the mathematical side, the study of non-self-adjoint Hamiltonians required the
development of new operator-theoretic tools.
Foundational contributions were made by Kato beginning in the 1950s
\cite{KT,KT2}, culminating in his monograph on perturbation theory~\cite{KT1}.
A decisive breakthrough in the rigorous treatment of complex energies was achieved
in 1971 by Aguilar and Combes~\cite{AgCom}, and independently by Balslev and
Combes~\cite{BaCom}, who introduced dilation analyticity and complex scaling methods.
These techniques allowed resonances to be identified as eigenvalues of
non-self-adjoint deformations of the Schr\"odinger operator.

Subsequently, these ideas were developed and refined throughout the 1970s and early
1980s in the works of Simon~\cite{SimBarr,SimBarr2}, Hunziker~\cite{Hunz}, and
Davies~\cite{Dav1}.
Their contributions clarified the analytic structure of resonances, highlighted
the intrinsic spectral instability of non-self-adjoint operators, and established
a rigorous framework for the spectral analysis of complex Schr\"odinger operators.
However, this body of work was primarily qualitative in nature and did not yet aim
at explicit quantitative bounds on eigenvalues.

\section{Eigenvalue bounds}\label{sec4}

Throughout this section, let $V:=V_++V_-$, where $V_+=\max\{V,0\}$ and $V_-:=\max\{-V,0\}$, and assume $V_+\in L_{\rm loc}^1$ and $V_-\in L_{\rm loc}^{d/2}$.
\subsection{Keller-type estimates} A first step toward quantitative control of the lowest eigenvalue in terms of integral norms of the potential goes back to Keller~\cite{K61}. In his 1961 paper, Keller considered the one-dimensional Schr\"odinger operator and studied the variational problem of minimizing the lowest eigenvalue under a prescribed integral constraint on the potential.

In modern notation, it is convenient to formulate the problem for the Schr\"odinger operator
\[
H_V=-\frac{d^2}{dx^2}+V(x)
\qquad\text{on }L^2(\mathbb R).
\]

Then Keller's result may be viewed as a bound of the form
\[
\lambda_1(V)\ge -C\,\|V_-\|_{L^p(\mathbb R)}^{\frac{2p}{2p-1}},
\qquad p\ge 1,
\]
where \(\lambda_1(V)<0\) denotes the lowest eigenvalue. Since \(\lambda_1(V)\) is negative, this is equivalently written as
\[
|\lambda_1(V)|\le C\,\|V_-\|_{L^p(\mathbb R)}^{\frac{2p}{2p-1}}.
\]
In the limiting case \(p=1\), Keller noted that the result reduces to earlier unpublished work of Larry Spruch. The minimizing potential found by Keller is of the type introduced by Epstein~\cite{E30}.

Keller's bound was the first to demonstrate that individual eigenvalues of
Schr\"odinger operators can be quantitatively controlled by Lebesgue norms of the
potential.
Moreover, the variational problem underlying Keller-type eigenvalue bounds is equivalent, via a duality principle, to the problem of determining sharp constants in suitable Sobolev interpolation inequalities, see e.g. \cite{CFL14,LT76}.

An important subsequent development was the emergence of more refined tools in
quantitative spectral theory for Schr\"odinger operators.
In particular, this led to the formulation and systematic use of several
fundamental principles:

\subsection{The Birman--Schwinger principle \cite{Birman1961,Birman1962,Schwinger1961}}\label{BS}
It states that if \(H_0\) is a self-adjoint operator and \(H=H_0+V\) is a
(possibly non-self-adjoint) perturbation, then for any
$\lambda\in\rho(H_0)$ the value $\lambda$ is an eigenvalue of $H$ if and
only if $-1$ is an eigenvalue of the associated Birman--Schwinger operator
$K(\lambda)=|V|^{1/2}(H_0-\lambda)^{-1}\operatorname{sgn}(V)|V|^{1/2}.$

\subsection{Schatten--von Neumann trace ideals \cite{Schatten1960,VonNeumann1932,Seir10}}

For Schr\"odinger operators, Schatten--von Neumann trace ideals enter naturally through the Birman--Schwinger principle. Indeed, if
\[
H=-\Delta+V,
\]
then spectral information on \(H\) can often be reduced to estimates on an associated Birman--Schwinger operator \(K(\lambda)\). In particular, Schatten norm bounds on \(K(\lambda)\) provide a quantitative way to control the discrete spectrum and, ultimately, eigenvalue sums.

Let \(\mathcal H\) be a Hilbert space and let \(T\colon \mathcal H\to\mathcal H\) be compact. Denote by $s_n(T)$, $n\in\mathbb N$,
the singular values of \(T\), that is, the eigenvalues of \(|T|=(T^*T)^{1/2}\), counted with multiplicity. For \(1\le p<\infty\), the Schatten class \(\mathfrak S_p(\mathcal H)\) is defined by
\[
\mathfrak S_p(\mathcal H)
:=
\left\{T\in\mathcal K(\mathcal H): \|T\|_{\mathfrak S_p}<\infty\right\},
\]
where
\[
\|T\|_{\mathfrak S_p}
:=
\left(\sum_{n=1}^\infty s_n(T)^p\right)^{1/p}.
\]
Thus, \(\mathfrak S_p\)-membership measures the summability of the singular values of \(T\).

In the applications considered here, one proves that the Birman--Schwinger operator \(K(\lambda)\) belongs to a suitable Schatten class and estimates
\[
\|K(\lambda)\|_{\mathfrak S_p}.
\]
Such bounds are the basic input for trace ideal methods and allow one to derive quantitative information on the eigenvalues of the corresponding Schr\"odinger operator.

\subsection{Lieb--Thirring-type inequalities} 
A key development was the work of Lieb and Thirring~\cite{LT75,LT76}, who proved bounds on moments of negative eigenvalues of Schr\"odinger operators $H$ in terms of Lebesgue norms of the real-valued potential. These results were crucial in providing a simple proof of the stability of matter; for a textbook introduction to this topic, see, e.g.,~\cite{LS10}.

More precisely, let \(H=-\Delta+V\) on \(L^2(\mathbb R^d)\), where \(V:\mathbb R^d\to\mathbb R\) is real-valued, and let \(\{\lambda_j\}\) denote the negative eigenvalues of \(H\), counted with multiplicity. Then
\[
\sum_{\lambda_j<0} |\lambda_j|^\gamma
\;\le\;
L_{\gamma,d}
\int_{\mathbb R^d} V_-(x)^{\gamma+d/2}\,dx,
\]
where \(V_-:=\max\{-V,0\}\), provided that
\[
\gamma \ge \frac12 \quad \text{if } d=1,\qquad
\gamma >0 \quad \text{if } d=2,\qquad
\gamma \ge 0 \quad \text{if } d\ge 3.
\]
Here \(L_{\gamma,d}>0\) is a constant depending only on \(\gamma\) and the dimension \(d\).

The cases \(\gamma>\frac12\) for \(d=1\) and \(\gamma>0\) for \(d\ge2\) were established by Lieb and Thirring~\cite{LT75,LT76}. In the endpoint case \(\gamma=0\) for \(d\ge3\), the left-hand side reduces to the number of negative eigenvalues; this case was proved independently by Cwikel \cite{MC}, Lieb \cite{Lieb1976Bounds}, and Rozenblum \cite{RozGV}, and is therefore also known as the Cwikel--Lieb--Rozenblum bound. The remaining critical case \(\gamma=\frac12\) in dimension \(d=1\) was proved by Weidl \cite{Weidl1996}.
Such inequalities establish a precise quantitative link between accumulation of eigenvalues of $H$ and $L^p$-integrability of the potential.


\section{Eigenvalue bounds for complex decaying potentials}\label{sec5}

Although originally formulated in the self-adjoint setting, the ideas described above provided the analytic foundation for later progress on non-self-adjoint Schr\"odinger operators. In particular, the Birman--Schwinger principle, Schatten class methods, and Lieb--Thirring-type inequalities have all been adapted and refined to treat complex-valued potentials and eigenvalues off the real axis.

These developments led to quantitative bounds on individual eigenvalues, eigenvalue sums, and the location of the discrete spectrum for non-self-adjoint Schr\"odinger operators. Important recent contributions in this direction were made by Frank~\cite{FR1,FR3}, Frank, Laptev, Lieb, and Seiringer~\cite{FRLALESR}, Frank, Laptev, and Weidl~\cite{Frank_Laptev_Weidl_2022}, and Frank and Simon~\cite{FRS2}.
Some of the earliest contributions in this area are due to Pavlov \cite{Pavlov1966,Pavlov1967,Pavlov1968}.

The first estimate of single, complex eigenvalues in terms of Lebesgue norms was obtained by Abramov, Aslanyan, and Davies~\cite{AAD}.
For complex-valued potentials on the line, they proved that
\begin{equation*}\label{abraskadav}
  \sup_{z\in\sigma(-\Delta+V)\setminus(0,\infty)} |z|^{1/2}
\le \frac12 \int_{\mathbb{R}} |V(x)|\,dx,  
\end{equation*}
where $\sigma(A)$ denotes the spectrum of an operator $A$. Further generalizations and related results concerning \eqref{abraskadav} were obtained by Davies and Nath in \cite{EBDJN}, and by Safronov in \cite{SOsum}.

\subsection{Deterministic Potentials on $\R^d$}

Following earlier work of Laptev and Safronov \cite{LASO,SO2009}, Frank \cite{FR1} established for complex-valued potentials the estimate
\begin{equation*}\label{Frankbounds}
   |z|^{q-d/2} \lesssim \int_{\mathbb{R}^d} |V(x)|^q\,dx 
\end{equation*}

for eigenvalues $z$ of $-\Delta+V$ whenever $\frac d2<q \le \frac12(d+1)$ for $d\ge2$. We refer to potentials $V$ in this regime as short-range. Earlier, Laptev and Safronov \cite{LASO} conjectured that this estimate should remain valid in the larger range $q\le d$. However, the short-range condition is in fact optimal: the above bound fails for $q > \frac12(d+1)$. Counterexamples for $z>0$ were constructed by Frank and Simon \cite{FRS2}, while counterexamples for $\operatorname{Im} z \neq 0$ were later obtained by B\"ogli and Cuenin \cite{CJCBS}, thereby disproving the conjecture.
More precisely, for \(\varepsilon>0\), let
\[
\chi_\varepsilon
:=
\mathbf{1}_{\{(x_1,x')\in \mathbb{R}\times\mathbb{R}^{d-1}:\ |x_1|<\varepsilon^{-1},\ |x'|<\varepsilon^{-1/2}\}}.
\]
Cuenin and Bögli constructed potentials \(V_\varepsilon\) satisfying
\[
|V_\varepsilon|\le \varepsilon\,\chi_\varepsilon
\]
such that
\[
z_\varepsilon:=1+i\varepsilon
\]
is an eigenvalue of the Schr\"odinger operator \(H_{V_\varepsilon}\). They show that for \(d\ge 2\) and \(q>(d+1)/2\),
\begin{equation}\label{counterex}
   \limsup_{\varepsilon\to 0}
|z_\varepsilon|^{\,q-\frac d2}\,
\|V_\varepsilon\|_{L^q(\mathbb{R}^d)}^{\,q}
=+\infty. 
\end{equation}
Hence, in this range of exponents, no uniform eigenvalue estimate of the type predicted by the short-range theory can hold along this family. In addition, it was proved that for \(d\ge 2\) and \(q\ge (d+1)/2\),
\[
\liminf_{\varepsilon\to 0}
\frac{
\operatorname{dist}(z_\varepsilon,\mathbb{R}_+)^{\,q-\frac{d+1}{2}}
|z_\varepsilon|^{1/2}
}{
\|V_\varepsilon\|_{L^q(\mathbb{R}^d)}^{\,q}
}
>0.
\]
In particular, this family also yields the corresponding lower-bound asymptotics at and above the critical exponent. The construction of the counterexample is motivated by the classical Knapp example from restriction theory, which demonstrates the sharpness of the Stein–Tomas theorem; see, for instance, \cite{Demeter_2020,CS}.

\subsection{Random Schr\"odinger operators on $\R^d$.}\label{randomschr}

Motivated by Bourgain's work \cite{BJ} on Schr\"odinger operators on $\Z^2$, randomization procedures can be used to improve bounds that are sharp in the deterministic setting, at the expense of allowing for a small exceptional set. In this context, one considers random Schr\"odinger operators of Anderson type, that is, operators of the form
\[
-\Delta+V_\omega,
\]
where \(V_\omega\) is obtained by randomizing a deterministic potential \(V\) at scale \(h>0\) according to
\[
V_\omega(x)=\sum_{j\in h\mathbb Z^d}\omega_j V(x)\mathbf{1}_Q((x-j)/h),
\qquad Q=[0,1)^d,
\]
and where the random variables \((\omega_j)_{j\in h\mathbb Z^d}\) are assumed to be independent, mean-zero Gaussian or symmetric Bernoulli random variables.

More precisely, although Frank's estimates \eqref{Frankbounds} are sharp, the short-range exponent can essentially be doubled with high probability. One of the results of Cuenin--Merz \cite{CJCMK} states that there exist constants $M_0,c>0$ such that for any $R,\lambda>0$, $0<h<R$, $|\varepsilon|\ll \lambda$, $q\le d+1$, $V\in L^q(\R^d)$ supported in a ball of radius $R$, and $M\ge M_0$, every eigenvalue $z=(\lambda+i\varepsilon)^2$ of $-\Delta+V_w$ satisfies
\begin{equation*}\label{randschr}
    \frac{\lambda^{2-\frac{d}{q}}}{\langle\lambda h\rangle^{{\frac{d}{2}}}\left(\log\langle\lambda R\rangle\right)^{\frac{7}{2}}}\le M\|V\|_{L^q(\R^d)},
\end{equation*}
except for $\omega$ in a set of measure at most $\exp(-cM^2)$.

In connection with the counterexample \eqref{counterex}, this implies that, under randomization at scale
\[
h\le\left(\varepsilon^{\frac{d+1}{2q}-1}\log\left(\frac{1}{\varepsilon}\right)^{-\frac{7}{2}}\right)^{\frac{2}{d}},
\]
the counterexample for $\frac{1}{2}(d+1)<q\le d+1$ is destroyed with high probability.

Estimates for sums of eigenvalues have also been investigated by Cuenin and Merz \cite{KMJCCKyoto} and by Safronov \cite{S2023}. 

\subsection{Deterministic Schr\"odinger operators on compact manifolds.}\label{jeanclaudecompact} Let 
 $M$ be a smooth compact manifold without boundary. In \cite{CJC}, Cuenin proved spectral bounds for the Schr\"odinger operator $-\Delta_g+V$, when $V\in L^q(M)$, which state that
\begin{equation*}\label{eq:deterministicmainresult}
    \operatorname{spec}(-\Delta_g+V) \subseteq \bigcup_{k\in\N_0}D(\lambda_k^2,C_{M,q,g}r_k), \cup \left\{z\in\C:\, |z|^{\frac12}(1+|z|)^{-\sigma(q)} \lesssim_{M,q,g}  \|V\|_{L^q(M)}\right\},  
\end{equation*}
where \(\{\lambda_k^2\}_{k=0}^\infty\) are the eigenvalues of \(-\Delta_g\), and
  $$r_k := \|V\|_{L^q(M)}\cdot(1+\lambda_k)^{2\sigma(q)},$$
and
\begin{align*}
  \label{eq:deterministicsigma}
  \sigma(q) :=
  \begin{cases}
    \frac{d}{2q}-\frac12 & \quad \text{if} \ \frac{d}{2}\leq q\leq\frac{d+1}{2}, \\
    \frac{d-1}{4q} & \quad \text{if} \ \frac{d+1}{2}\leq q \leq \infty,
  \end{cases}
\end{align*}
where $\lambda_k^2$ are the eigenvalues of the free Laplacian $\Delta_g$. See more details in Section \ref{freelaplacecompactmanifolds}.

In particular, this estimate is optimal for Zoll manifolds, i.e. compact Riemannian manifolds for which every geodesic is closed and all unit-speed geodesics have the same period. The prototypical example is the round sphere \((S^n,g_{\mathrm r})\), equipped with its standard round metric \(g_{\mathrm r}\) of constant curvature. Its geodesics are exactly the great circles, so after normalization all unit-speed geodesics are closed and have common period \(2\pi\); see \cite[ch.~8.4.1]{z17} for further details.

\subsection{Random Schr\"odinger operators on compact manifolds}
For compact manifolds, one can combine methods analogous to those used in the proofs of the results in Section~\ref{randomschr} with Sogge's spectral bounds from \cite{CS,SOGGE1988} to obtain improvements to \eqref{eq:deterministicmainresult} that are comparable to the improvement from \eqref{Frankbounds} to \eqref{randschr}. This result will appear in forthcoming work of Cuenin, Merz, and the author.

\section{Generalization to fractional Laplacians}\label{sec6}

The proofs in Cuenin's work \cite{CJC} are based on spectral bounds, which apply to a rather general class of operators, and the works of Dos Santos Ferreira, Kenig, and Salo \cite{DosSantosFerreiraKenigSalo2014}, and of Frank and Schimmer \cite{FRANKSchimmer2017}. 
This suggests the possibility of extending the above results beyond the classical Schr\"odinger setting. It is therefore natural to investigate whether the above-mentioned results of Cuenin admit an analogous extension to the fractional Laplacian setting. 
Further important related results are due to Kenig, Ruiz and Sogge \cite{KenigRuizSogge1987}, Shen and Zhao \cite{ShenZhao08}, and Bourgain, Shao, Sogge, and Yao \cite{BourgainShaoSoggeYao2015}.

The following theorem is due to the author and will appear in forthcoming work. 
We include here only a brief outline of the proof; full details will be provided in that forthcoming paper.

\begin{theorem}\label{psithm1}
     Let \((-\Delta_g)^{\alpha/2}\) be the fractional Laplacian on a closed Riemannian manifold \((M,g)\) of dimension \(d \ge 2\). Let 
$\frac{2d}{d+1} \le \alpha \le d$,
and
\[
\frac{d}{\alpha} < q \le \frac{2d}{d-\alpha}
\quad \text{if } d>\alpha,
\qquad \text{or} \qquad
1 \le q < \infty
\quad \text{if } d=\alpha.
\]
     Then there exists a constant \(C = C(M, g, q, \alpha)\) such that for all \(V \in L^q(M)\),
  \begin{equation}\label{result}
    \mathrm{spec}((-\Delta_g)^{\alpha/2} + V)
\subset \bigcup_{k=0}^\infty D\bigl(\lambda_k^\alpha,\; C\,r_k \bigr)
\;\cup\; \Bigl\{ z \in \mathbb{C} : |z|^{1-\frac1\alpha} (1+|z|)^{-\frac{2\sigma(q)}{\alpha}} \le C \,\|V\|_{L^q(M)} \Bigr\},  
  \end{equation}
where \(\{\lambda_k^\alpha\}_{k=0}^\infty\) are the eigenvalues of \((-\Delta_g)^{\alpha/2}\), and
\[
r_k := \|V\|_{L^q(M)}(1 + \lambda_k)^{2\sigma(q)},
\]
with
\[
\sigma(q) :=
\begin{cases}
\dfrac{d}{2q} - \dfrac12, & \tfrac{d}{\alpha} < q \le \tfrac{d+1}{2}, \\
\dfrac{d-1}{4q}, & \tfrac{d+1}{2} \le q \le \infty.
\end{cases}
\]
\end{theorem}


\begin{rem}
    We believe that the structural framework developed in \cite{CUENIN2024110214,CJC,FRANKSchimmer2017,KU15} allows to extend these results also to positive self-adjoint elliptic pseudodifferential operators of arbitrary positive order, under suitable geometric assumptions such as curvature conditions; see, for instance, Sogge's monograph \cite{CS}.
\end{rem}


\section{Proof idea of Theorem \ref{psithm1}}\label{sec7}
We follow Cuenin \cite{CJC} and adapt the proofs to the general case of fractional Laplace operators, see Section \ref{appA} or \cite{CS} for the precise definitions.

\begin{proof} Complex powers are taken with respect to the principal logarithm,
that is,
\[
    w^\alpha:=\exp(\alpha \log w),
    \qquad \text{Arg}  (w)\in(-\pi,\pi).
\]
Define
\[
    \Gamma
    :=
    \left\{
        (\lambda+i)^\alpha:\lambda\geq \cot\!\left(\frac{\pi}{\alpha}\right)
    \right\}
    \cup
    \left\{
        (\lambda-i)^\alpha:\lambda\geq \cot\!\left(\frac{\pi}{\alpha}\right)
    \right\}.
\]
The two arcs meet at the point
\[
    z_\ast
    :=
    \left(\cot\!\left(\frac{\pi}{\alpha}\right)+i\right)^\alpha
    =
    \left(\cot\!\left(\frac{\pi}{\alpha}\right)-i\right)^\alpha
    =
    -\sin\!\left(\frac{\pi}{\alpha}\right)^{-\alpha}.
\]
Thus \(\Gamma\) is a continuous curve in \(\mathbb C\). We define
\(\Xi_0\) to be the connected component of
\(\mathbb C\setminus\Gamma\) containing the half-line
\[
    (-\infty,z_\ast).
\]
Finally, we set
\[
    \Xi:=\Xi_0\cup\Gamma.
\]
Geometrically, \(\Xi_0\) is the exterior region bounded by the two arcs
forming \(\Gamma\); see Figure~\ref{fig}.

\begin{figure}
    \centering
    \includegraphics[width=0.75\linewidth]{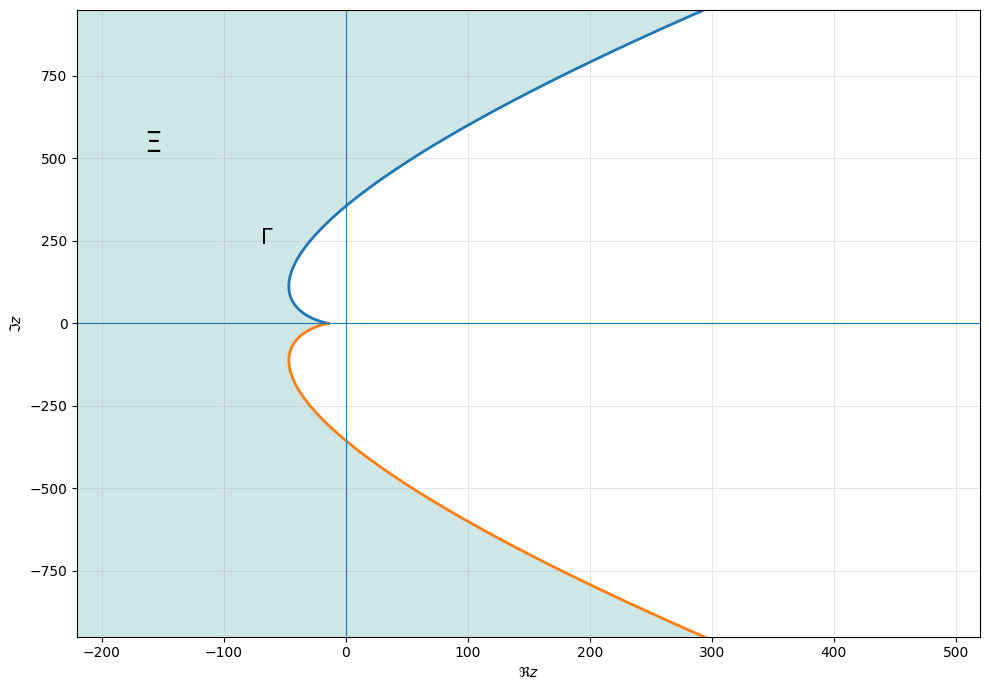}
    \caption{Region $\Xi$ and boundary $\Gamma$ for $\alpha=5$.}
    \label{fig}
\end{figure}

    As in \cite{CJC} Theorem \ref{psithm1} follows from the Birman-Schwinger principle, and
    \begin{subequations}
        \begin{align}
        \label{firsteq}
\|((-\Delta_g)^{\alpha/2} - z)^{-1}\|_{L^{p}(M)\to L^{p'}(M)}
&\lesssim |z|^{\frac{2\nu(p')+1}{\alpha}-1},
\qquad  &\text{for}\ z \in\Xi,
\\ \label{secondeq}
\|((-\Delta_g)^{\alpha/2} - z)^{-1}\|_{L^{p}(M)\to L^{p'}(M)}
&\lesssim d(z)^{-1}(1+|z|)^{\frac{2\nu(p')}{\alpha}}
 &\text{for}\ z\notin\Xi,
\end{align} 
where 
$$\nu(p')=\sigma(q),\quad d(z)=\operatorname{dist}\left(z,\operatorname{spec}\left((-\Delta_g)^{\frac\alpha2}\right)\right)$$
    \end{subequations}
    and 
    $$\frac{1}{p}+\frac{1}{p'}=1\ \quad \text{and}\quad \frac1q=\frac1p-\frac{1}{p'}.$$

\textbf{Step 1. Complex extension:} The first estimate \eqref{firsteq} is known to be true only in $\Gamma$, by Cuenin \cite[Theorem 2.25]{CUENIN2024110214}. A Phragm\'en-Lindelöf argument is sufficient to extend the equation to all of $\Xi$.

\textbf{Step 2. Bounds near the singularities:}
The estimate in \eqref{secondeq} can be obtained by an argument analogous to that in \cite{CJC}. 
The required ingredients are the relevant Sobolev embeddings, see e.g. \cite[Chapter 13]{taylor2010partial}, and Sogge's spectral cluster bounds, see \cite[Chapter 5]{CS} or \cite{SOGGE1988}, which allow one to adapt several parts of the proof in \cite{CJC}. 
Moreover, the necessary estimates on the asymptotic behavior of the associated complex quantities are ensured by the appropriate choice of the parabola-like contour $\Gamma$.

\end{proof}

\appendix
\section{Free Laplacian on compact manifolds}\label{app}

Since our main goal is to study the spectral properties---in particular the eigenvalues---of the Schr\"odinger operator
\[
H=-\Delta+V
\]
on compact manifolds, we begin by briefly reviewing the unperturbed case. We also sketch the main ideas of the proof. For further details and comprehensive treatments, see, for example, \cite{KT1,RSII,RSIV,tg}.

\subsection{Free Laplacian on compact manifolds}\label{freelaplacecompactmanifolds}
Let $(M,g)$ be a compact smooth Riemannian manifold without boundary, 
and let $H = -\Delta_g$ act on $L^2(M)$ with form domain $H^1(M)$ 
and operator domain $H^2(M)$. 
This operator is non-negative and self-adjoint, as follows from the closed, 
lower-bounded quadratic form and the Friedrichs construction. 
By elliptic regularity, for every $u \in H^2(M)$ we have
\[
\|u\|_{H^2} \lesssim \|Hu\|_{L^2} + \|u\|_{L^2}.
\]
Hence, for any fixed $\lambda > 0$, if $f \in L^2(M)$ and $u$ satisfies 
$(H+\eta)u = f$, then
$(H+\eta)^{-1} : L^2(M) \to H^2(M)$
is bounded. 
Since $M$ is compact, the embedding $H^2(M) \hookrightarrow L^2(M)$ 
is compact by the Rellich--Kondrachov theorem 
(see, e.g., \cite[Chapter~5.7]{LE}). 
Consequently, $(H+\eta)^{-1}$ is a compact operator. 
By the spectral theorem for compact operators, its spectrum consists of 
at most countably many eigenvalues $\mu_n$ that can accumulate only at zero. 
This directly implies that $H$ possesses at most countably many eigenvalues 
$\lambda_n = \mu_n^{-1} - \eta$. 
More generally, the same reasoning---relying on elliptic regularity and the Rellich--Kondrachov theorem---applies to elliptic self-adjoint pseudodifferential operators of positive order, such as the fractional Laplacian \(\flap\); see, for example, \cite{HL3,HL4}.

\subsection{Laplacian on compact manifold with bounded potential} Finally, we remark that for $V\in L^\infty$ under the same assumption as Theorem \ref{psithm1}, $(-\Delta_g)^{\alpha/2}+V$ can be defined as an $m-$sectorial operator by quadratic form methods, see \cite{FR3} for a proof and \cite{KT1} for the general theory.

\vspace{0,4cm}

\textbf{Acknowledgments}
This research was funded in whole, or in part, by the Austrian Science Fund (FWF)  [Grant-DOI 
10.55776/P35322, 10.55776/PAT5120424, and 
10.55776/PAT4632823]. 
The author is grateful to Christoph Aistleitner, Andrei Shubin, Petr Siegl, Jean-Claude Cuenin, and Konstantin Merz for many helpful discussions. He would like to express his special thanks to Konstantin Merz for his particularly valuable support and insights.

\bibliographystyle{abbrv}
\bibliography{main}







\end{document}